\documentclass[reqno, letterpaper, oneside]{article}
\usepackage{amsmath,amsfonts,amssymb,amsthm}
\usepackage[final]{graphicx}
\usepackage[usenames,dvipsnames]{color}

\usepackage{bbm}
\usepackage{setspace}
\usepackage{geometry}
\usepackage{enumerate}

\usepackage[colorlinks=false,,linktoc=allbookmarksopen=true,linktocpage,pdftex]{hyperref}
\usepackage[pdftex]{bookmark}

\usepackage{comment}
\usepackage[normalem]{ulem}
\usepackage{authblk}
\usepackage{amsmath}
\usepackage{amsthm}
\usepackage{amsfonts}
\usepackage{amssymb}
\usepackage[usenames,dvipsnames]{color}
\theoremstyle{plain}
\newtheorem{theorem}{Theorem}[section]

\newtheorem{conjecture}[theorem]{Conjecture}
\newtheorem{lemma}[theorem]{Lemma}

\newtheorem{problem}{Problem}
\makeatletter
\newcommand{\vast}{\bBigg@{4}}
\newcommand{\Vast}{\bBigg@{5}}
\makeatother
\definecolor{bulgarianrose}{rgb}{0.28, 0.02, 0.03}
\definecolor{gray}{rgb}{0.5, 0.5, 0.5}

\theoremstyle{definition}

\usepackage{hyperref}
\makeatletter
\def\namedlabel#1#2{\begingroup
    #2%
    \def\@currentlabel{#2}%
    \phantomsection\label{#1}\endgroup
}

\usepackage{enumerate}

\newcommand{\refL}[1]{Lemma~\ref{#1}}

\newcommand\E{\operatorname{\mathbb E{}}}
\renewcommand\Pr{\operatorname{\mathbb P{}}}

\newcommand\RR{{\mathbb R}}

\newcommand{\cA}{\mathcal{A}}
\newcommand{\cB}{\mathcal{B}}
\newcommand{\cC}{\mathcal{C}}
\newcommand{\cD}{\mathcal{D}}
\newcommand{\cG}{\mathcal{G}}
\newcommand{\cU}{\mathcal{U}}

\newcommand{\cT}{\mathcal{T}}
\newcommand{\fC}{\mathfrak{C}}
\newcommand{\fS}{\mathfrak{S}}

\newcommand{\indic}[1]{\mathbbm{1}_{\{{#1}\}}}

\newcommand\bigpar[1]{\bigl(#1\bigr)}
\newcommand\Bigpar[1]{\Bigl(#1\Bigr)}
\newcommand\biggpar[1]{\biggl(#1\biggr)}

\newcommand\bigsqpar[1]{\bigl[#1\bigr]}
\newcommand\Bigsqpar[1]{\Bigl[#1\Bigr]}
\newcommand\biggsqpar[1]{\biggl[#1\biggr]}

\newcommand\bigcpar[1]{\bigl\{#1\bigr\}}
\newcommand\Bigcpar[1]{\Bigl\{#1\Bigr\}}
\newcommand\biggcpar[1]{\biggl\{#1\biggr\}}
\newcommand\Biggcpar[1]{\Biggl\{#1\Biggr\}}

\newcommand\ceil[1]{\lceil#1\rceil}
\newcommand\bigceil[1]{\bigl\lceil#1\bigr\rceil}

\newcommand\floor[1]{\lfloor#1\rfloor}

\newcommand{\phat}{\hat{p}}
\newcommand{\Gnp}{G_{n,p}}
\newcommand{\Gnpp}[1]{G_{n,{#1}}}
\newcommand{\Gnph}{G_{n,\phat}}

\newcommand{\cupdot}{\mathbin{\mathaccent\cdot\cup}}

\let\OLDthebibliography\thebibliography
\renewcommand\thebibliography[1]{
  \OLDthebibliography{#1}
  \setlength{\parskip}{0pt}
  \setlength{\itemsep}{0pt plus 0.3ex}
}

\title{The jump of the clique chromatic number of random graphs} 

\author{Lyuben Lichev\thanks{Ecole Normale Sup\'erieure de Lyon, Lyon, France. 
E-mail: {\tt lubetidobrilov@gmail.com}.}
 \ and 
Dieter Mitsche\thanks{Institut Camille Jordan, Univ. Lyon 1, Lyon, France and Univ.\ Jean Monnet, Saint-Etienne, France. 
E-mail: {\tt dmitsche @unice.fr}.
Research partially supported by grant GrHyDy ANR-20-CE40-0002 and by IDEXLYON of Universit\'{e} de Lyon (Programme Investissements d'Avenir ANR16-IDEX-0005).}
 \ and 
Lutz Warnke\thanks{School of Mathematics, Georgia Institute of Technology, Atlanta GA~30332, USA. 
E-mail: {\tt warnke@math.gatech.edu}.
Research partially supported by NSF grant DMS-1703516, NSF~CAREER grant~DMS-1945481, and a Sloan Research Fellowship.}}

\date{May 25, 2021} 

\begin{document}

\maketitle
 
\begin{abstract}
The clique chromatic number of a graph is the smallest number of colors in a vertex coloring 
so that no maximal clique is monochromatic. 
In~2016 McDiarmid, Mitsche and Pra{\l}at noted that around~$p \approx n^{-1/2}$ 
the clique chromatic number of the random graph~$G_{n,p}$ changes by~$n^{\Omega(1)}$ 
when we increase the edge-probability~$p$ by~$n^{o(1)}$, 
but left the details of this surprising phenomenon as an open~problem. 

We settle this problem, i.e., resolve the nature of this polynomial `jump' of the clique chromatic number 
of the random graph~$G_{n,p}$ around~edge-probability~$p \approx n^{-1/2}$.  
Our proof uses a mix of approximation and concentration arguments, 
which enables us to (i)~go beyond Janson's inequality used in previous work
and (ii)~determine the clique chromatic number of~$G_{n,p}$ up to logarithmic~factors for any~edge-probability~$p$. 
\end{abstract}

\section{Introduction}
The chromatic number is one of the central topics in random graph theory, 
which has repeatedly advanced the field.  
Studying variants of such fundamental parameters is not only a way to sharpen our tools and techniques, 
but also motivated by our desire to better understand the underlying nature of their behavior, which can be richer than one might suspect. 
For example, the usual chromatic number of a random graph increases smoothly with the average degree. 
By~contrast, this paper concerns a clique based variant of the chromatic number that exhibits a very different behavior: 
it can `jump' by~$n^{\Omega(1)}$ when we increase the average degree by~$n^{o(1)}$; cf.~Figure~\ref{fig:upper_bound}.

\begin{figure}[t]
\centering
\includegraphics[height=4.5cm]{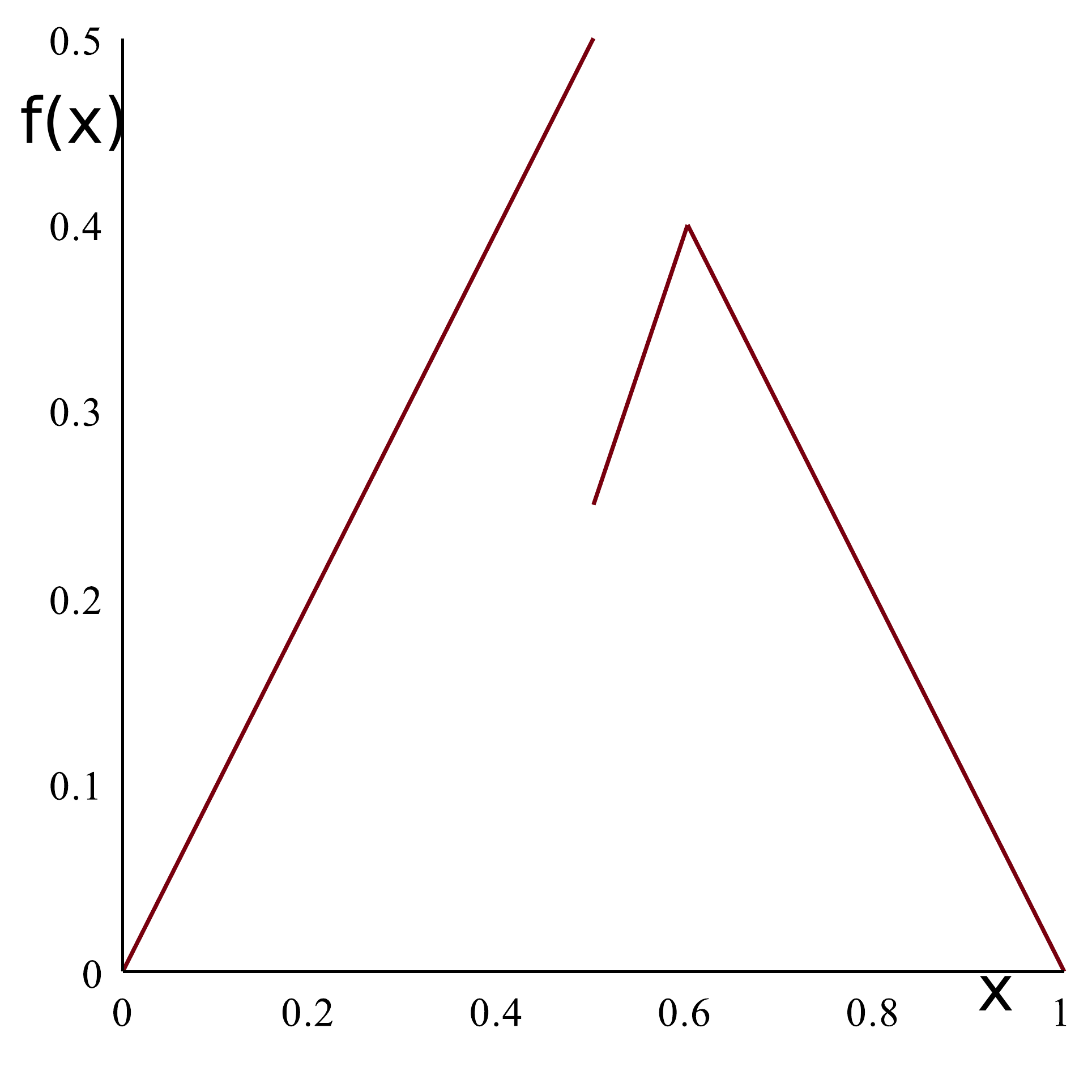} 
\vspace{-1.0em}
\caption{McDiarmid, Mitsche and Pra{\l}at showed that if the average degree is~${np=n^{x+o(1)}}$ for some constant~${x \in (0, 1) \setminus \{1/2\}}$, 
then the typical clique chromatic number of the random graph~$\Gnp$ is~${\chi_c(\Gnp)=n^{f(x)+o(1)}}$
for the 
function~${f(x):=\max\{x\indic{x<1/2},\min\{(3x-1)/2,1-x\}\}}$ plotted above. 
Our main result \mbox{Theorem~\ref{thm:main1}} zooms in on the 
discontinuity point~${x=1/2}$, 
and explains the polynomial `jump' of the clique chromatic number~$\chi_c(\Gnp)$ 
from~$n^{1/2+o(1)}$ down to~$n^{1/4+o(1)}$ around average degree~$np=n^{1/2+o(1)}$.%
\label{fig:upper_bound}}
\end{figure}


The \emph{clique chromatic number} of a graph~$G$, denoted by $\chi_c(G)$,
is the smallest number of colors needed to color the vertices of~$G$ so that no inclusion-maximal 
clique is monochromatic (ignoring isolated vertices). 
Writing~$\chi(G)$ for the usual chromatic number, we always have~$\chi_c(G) \le \chi(G)$, with equality for triangle-free graphs (among others). 
But~$\chi_c(G)$ can be much smaller than~$\chi(G)$, since~$2=\chi_c(K_n)< \chi(K_n)=n$ for~$n \ge 3$.   
An important conceptual difference is monotonicity with respect to taking 
subgraphs: 
while~$H \subseteq G$ always implies~$\chi(H) \le \chi(G)$, this is no longer true\footnote{%
For an example with~$H \subseteq G$ and~$\chi_c(H) > \chi_c(G)$, 
we connect a new vertex to all vertices of a triangle-free graph~$H$ with~${\chi(H)>k}$; 
the resulting graph~$G \supseteq H$ clearly satisfies~${\chi_c(G)=2}$, whereas triangle-freeness ensures that~${\chi_c(H)=\chi(H)>k}$.} 
for the clique chromatic~number; see~\cite{AST1991,DSSW1991,MS1999,CPTT2016,JMRS} for many additional structural~results.  
Furthermore, the basic algorithmic problem of deciding~$\chi_c(G) \le 2$ is 
\mbox{NP-complete}~\cite{KT2002,BGGPS2004,Marx2011}, 
illustrating that problems which are easy for the usual chromatic number can become difficult for the clique chromatic~number.

Around~2016, McDiarmid, Mitsche and Pra{\l}at~\cite{MMP} initiated the study of the clique chromatic number~$\chi_c(\Gnp)$ of the $n$-vertex binomial random graph~$\Gnp$ with edge-probability~$p=p(n)$.
They showed that if the average degree of~$\Gnp$ is~${np=n^{x+o(1)}}$ for some constant~${x \in (0, 1) \setminus \{1/2\}}$, 
then with high probability\footnote{As usual, we say that an event holds~\emph{whp} (with~high~probability) if it holds with probability tending to~$1$ as~$n\to \infty$.} (whp) 
the clique chromatic number is~${\chi_c(\Gnp)=n^{f(x)+o(1)}}$ for the discontinuous function~$f(x)$ plotted in Figure~\ref{fig:upper_bound}. 
In particular, around the discontinuity point~$x=1/2$ an increase of the average degree 
by~$n^{o(1)}$ can thus decrease~${\chi_c(\Gnp)}$ by~$n^{\Omega(1)}$, 
but the finer details of this polynomial `jump' of~${\chi_c(\Gnp)}$ 
remained the main open problem (see~\cite[Section~5]{MMP}).
Indeed, the largest gaps in~\cite{MMP} occur around average degree~$np=n^{1/2+o(1)}$, 
where their whp estimates~$n^{1/4+o(1)} \le \chi_c(\Gnp) \le n^{1/2+o(1)}$ 
leave the nature of the surprising jump of~${\chi_c(\Gnp)}$ completely~open.
%

\pagebreak[2]

\subsection{Main contributions} 
In this paper we resolve the nature of the polynomial `jump' of the clique chromatic number~$\chi_c(\Gnp)$ 
depicted in Figure~\ref{fig:upper_bound}. 
More concretely, Theorem~\ref{thm:main1} zooms in on the jump around~${p = n^{-1/2+o(1)}}$, 
and shows that the steep transition  
from ${\chi_c(\Gnp)}={\Theta(np/\log(np))}$ for~${p \ll n^{-1/2}}$ 
down to ${\chi_c(\Gnp)}={\Theta(p^{3/2}n/\sqrt{\log n})}$ for ${p \gg n^{-1/2}\sqrt{\log n}}$ 
arises due to an exponential factor of~$e^{-np^2}$ that was missing in previous work.
This in particular settles the main open problem for~$\chi_c(\Gnp)$  
due to McDiarmid, Mitsche and Pra{\l}at (see~\cite[Section~5]{MMP}). 
\begin{theorem}[Main result]\label{thm:main1} %
If the edge-probability~$p=p(n)$ satisfies~${n^{-0.6} \le p \le n^{-0.4}}$, 
then with high~probability the clique chromatic number of the random graph~$\Gnp$ satisfies 
\begin{equation}\label{eq:main1}
    \chi_c\bigpar{\Gnp} =  \tilde{\Theta}\biggpar{\max\biggcpar{\frac{e^{-np^2}np}{\log(np)}, \: \frac{p^{3/2}n}{\sqrt{\log n}}}},
\end{equation}
where the $\tilde{\Theta}$-notation suppresses polylogarithmic factors, i.e., extra factors of form~$(\log n)^{O(1)}$.
Furthermore, in~\eqref{eq:main1} the extra polylogarithmic factors are only needed when $1-o(1) \le 4np^2/\log n \le 2+o(1)$. 
\end{theorem}
The steep transition of the clique chromatic number~${\chi_c(\Gnp)}$ around~$p \approx n^{-1/2}$ 
is heuristically closely linked to the fact that~${\chi_c(G) \le \chi(G)}$ holds with equality for triangle-free graphs~$G$ (for which all edges are inclusion-maximal cliques). 
Indeed, it is well-known that `few' edges of~$\Gnp$ are in triangles when~${p \ll n^{-1/2}}$, 
in which case we loosely expect that whp~$\chi_c(\Gnp) \approx \chi(\Gnp) ={\Theta(np/\log(np))}$, see also~\cite[Theorem~1.3]{MMP}.
Furthermore, `most' edges of~$\Gnp$ are in triangles when~${p \gg n^{-1/2}}$, 
in which case we loosely expect that whp~$\chi_c(\Gnp) \ll \chi(\Gnp) ={\Theta(np/\log(np))}$. 
This discussion makes it plausible that around~$p=n^{-1/2+o(1)}$
 the difference between~$\chi_c(\Gnp)$ and~$np/\log(np)$ should be controlled by the 
probability~$(1-p^2)^{n-2} \approx e^{-np^2}$ that an edge is not contained in a triangle (i.e.,~forms an inclusion-maximal clique), 
as made precise by Theorem~\ref{thm:main1}.
Our proof obtains the key factor~$e^{-np^2}$ in~\eqref{eq:main1} using a delicate 
mix of approximation and concentration arguments, 
which enables us to go beyond Janson's inequality used in previous~work.

Combining Theorem~\ref{thm:main1} with previous work~\cite{MMP} we now know, for the first time, 
the typical value of~$\chi_c(\Gnp)$ up to~$n^{o(1)}$ factors for any~${p=p(n)}$. 
To improve our understanding of the clique chromatic number, 
it is desirable to reduce these gaps further, 
and our proof techniques are powerful enough to achieve this. 
Namely, 
Theorem~\ref{thm:main2} determines the typical value of~$\chi_c(\Gnp)$ up to logarithmic factors for any edge-probability~${p=p(n)}$.  
\begin{theorem}\label{thm:main2}
There is a constant~$C \ge 1$ such that, for any edge-probability~$p=p(n) \in [0,1]$, 
the clique chromatic number of the random graph~$\Gnp$ with high probability satisfies 
\begin{equation}\label{eq:main2}
\frac{\displaystyle\chi_c(\Gnp)}{\max\biggcpar{\displaystyle 1, \; \displaystyle\frac{e^{-np^2}np}{\log n}, \; \min\biggcpar{\displaystyle \frac{p^{3/2}n}{\sqrt{\log n}}, \: \displaystyle \frac{1}{p}}}} \: \in \: \Bigl[ \bigl(C \log n\bigr)^{-1} , \; C \log n\Bigr].
\end{equation}
\end{theorem}
In concrete words, Theorem~\ref{thm:main2} refines Figure~\ref{fig:upper_bound},  
by revealing how the clique chromatic number~$\chi_c(\Gnp)$ of the random graph~$\Gnp$ 
changes its typical behavior around~${p=n^{-1+o(1)}}$, ${p=n^{-1/2+o(1)}}$ and~${p=n^{-2/5+o(1)}}$.  
For certain specific ranges of~$p=p(n)$ sharper bounds than~\eqref{eq:main2} follow from Theorem~\ref{thm:main1} 
and the results in~\cite{MMP}, 
as well as the recent work of Alon and Krivelevich~\cite{AK17} and Demidovich and Zhukovskii~\cite{DK}; 
see also Section~\ref{sec:conclusion} for a conjecture that further refines the bounds of Theorem~\ref{thm:main2}.

\subsection{Organization of the paper}
The remainder of this paper is organized as follows. 
In the next subsection we state our main technical results, 
which give lower and upper bounds on the clique chromatic number~$\chi_c(\Gnp)$. 
Section~\ref{sec: lower bound}~and~\ref{sec: upper bounds} contain the proofs of these technical results, 
which in Section~\ref{sec:technical} are then used to deduce Theorem~\ref{thm:main1} and~\ref{thm:main2} via several case distinctions. 
The final Section~\ref{sec:conclusion} contains some brief concluding remarks and conjectures.

\subsection{Main technical results}\label{sec:maintechnical}
The main technical result of this paper is Theorem~\ref{thm:lower}, 
whose lower bound on the clique chromatic number~$\chi_c(\Gnp)$ 
intuitively comes from inclusion-maximal $k$-vertex cliques, see Section~\ref{sec: lower bound}.   
In Section~\ref{sec:technical} we will show that the lower bounds in Theorem~\ref{thm:main1} and~\ref{thm:main2} both follow from inequality~\eqref{eq:thm:lower:chi} 
by optimizing over all valid choices of~${k=k(n) \ge 2}$. 
For Theorem~\ref{thm:main1} it will be crucial to have the~${(1-p^k)^{n/(k-1)}} \approx {e^{-np^k/(k-1)}}$ factor in~\eqref{eq:thm:lower:chi} that was missing in previous work.
To cover the range~${p=n^{-o(1)}}$ in Theorem~\ref{thm:main2}, it will also be important to not impose any upper bound on~${k=k(n) \ge 2}$ beyond assumptions~\eqref{eq:thm:lower:p:upper:0}--\eqref{eq:thm:lower:p:upper}. 
Much of our proof efforts in Section~\ref{sec: lower bound} are devoted towards establishing these features, 
which are key conceptual differences to previous work~\cite{MMP}. 
\begin{theorem}[Main technical result]\label{thm:lower}
There is a constant~$C \ge e$ such that the following holds for any~$\tau \in (0,1)$. 
If an integer~${k=k(n) \ge 2}$ and the edge-probability~${p=p(n)}$ satisfy the two assumptions  
\begin{gather}
\label{eq:thm:lower:p:upper:0}
n^{-1+\tau} \; \le \; p \; \le \; (6 C)^{-1} \cdot \bigsqpar{k^2 \log(1/p) np}^{-1/2(k-1)} \cdot (1-p^k)^{n/(k-1)} , \\
\label{eq:thm:lower:p:upper}
\max\Bigcpar{np^2, \: p^{-(k-2)/2}} \cdot (1-p^k)^{n(k-2)/(k-1)} \; \ge \; \indic{k \ge 3} k^5 \log(1/p)  ,
\end{gather}
then with high~probability the clique chromatic number of the random graph~$\Gnp$ satisfies 
\begin{equation}\label{eq:thm:lower:chi}
    \chi_c\bigpar{\Gnp} \; \ge \; \frac{1}{C^2} \cdot \min\Biggcpar{\frac{1}{p}, \: \frac{p^{k/2}n}{\bigsqpar{k!k\log(1/p)}^{1/(k-1)} }} \cdot (1-p^k)^{n/(k-1)}.
\end{equation}
\end{theorem}

We complement Theorem~\ref{thm:lower} with the following upper bound on the clique chromatic number~$\chi_c(\Gnp)$, 
which in Section~\ref{sec:technical} will be used (together with previous work) to establish the upper bounds in Theorem~\ref{thm:main1} and~\ref{thm:main2}. 
For Theorem~\ref{thm:main1} it will be crucial to have the~$e^{-np^2}$ factor in~\eqref{eq:thm:upper:chi:1} that was missing in previous work.
We mainly included part~\eqref{pt 3'} of Theorem~\ref{thm:upper} to demonstrate that that one can sharpen part~\eqref{pt 1'} for certain ranges of~${p=p(n)}$: 
in particular, for~${1+\Omega(1) \le 2np^2/(\log n) \le 6}$ we infer from~\eqref{eq:thm:upper:chi:2} that whp~${\chi_c(\Gnp)}={O(p^{3/2}n/\sqrt{\log n)}}$, 
which sharpens~\eqref{eq:thm:upper:chi:1} and also allows us to weaken the corresponding assumption~${2np^2/(\log n) \ge 4}$ of~\cite[Theorem~3.1]{MMP}. 
\begin{theorem}\label{thm:upper}
There there are constants~$C_1,C_2>0$ such that the following holds.%
\vspace{-0.5em}\begin{enumerate}[(i)]
\partopsep=0pt \topsep=0pt \parskip0pt \parsep0pt \itemsep0.25em
    \item\label{pt 1'} If~$(\log n) n^{-2/3}\le p\le n^{-1/3}/\log n$, then with high probability 
\begin{equation}\label{eq:thm:upper:chi:1}
    \chi_c\bigpar{\Gnp} \; \le \; C_1 \cdot \max\biggcpar{\frac{e^{-np^2}np}{\log n}, \: \frac{p^{3/2}n}{\sqrt{\log n}}} \cdot \frac{\log n}{\log \Gamma},
\end{equation}
where~$\Gamma := \max\bigl\{e^{-np^2}p^{-1/2}\sqrt{\log n}, \: \log n\bigr\}$ satisfies~${\log \log n \le \log \Gamma = O(\log n)}$.  
    \item\label{pt 3'} If~$\sqrt{(2^{-1}\log n + \log \log n)/n}\le p\le \sqrt{3 (\log n)/n}$, then with high probability 
\begin{equation}\label{eq:thm:upper:chi:2}
    \chi_c\bigpar{\Gnp} \; \le \; C_2 \cdot \frac{p^{3/2}n}{\sqrt{\log n}} \cdot \frac{\log n}{\xi},
\end{equation}
where~$\xi := {2np^2-(\log n+\log\log n)}$ satisfies~${\log \log n \le \xi < 5 \log n}$.%
\end{enumerate}
\end{theorem}

\pagebreak[2]

\section{Lower bound: proof of Theorem~\ref{thm:lower}}\label{sec: lower bound} %
In this section we prove Theorem~\ref{thm:lower}, i.e., our main lower bound result.
For suitable choice of~$s$, our proof strategy will be to show that in the random graph~$\Gnp$ 
whp any vertex set~$S \subseteq [n]$ of size~$|S|=\ceil{s}$ contains at least one copy of a $k$-vertex clique~$K_k$ 
whose vertex-set~$\{v_1,\ldots, v_k\} \in \binom{S}{k}$ that has no common neighbor outside of~$S$ 
(i.e., for which there exists no vertex $w \in [n] \setminus S$ that is simultaneously adjacent to all the vertices~$v_j$).
This implies that every set of~$\ceil{s}$ vertices contains an inclusion-maximal clique, 
so that the color classes of any valid clique coloring can have at most~$\floor{s} \le s$ vertices.
Hence~$\chi_c(\Gnp)\ge n/s$, which for the `size'~parameter
\begin{equation}\label{def:s}
s \; := \; C \cdot \max \Bigcpar{np, \: p^{-k/2}\bigsqpar{k!k\log(1/p)}^{1/(k-1)}} \cdot (1-p^k)^{-n/(k-1)} 
\end{equation}
establishes the desired inequality~\eqref{eq:thm:lower:chi} 
with room to~spare (deferring our choice of the constant~$C \ge e$).
 
Note that the trivial inequality~${\chi_c(\Gnp) \ge 1}$ already implies the desired lower bound~\eqref{eq:thm:lower:chi}
when the right-hand side of~\eqref{eq:thm:lower:chi} is at most~one. 
In our proof of Theorem~\ref{thm:lower} we henceforth can thus safely assume~that 
\begin{equation}\label{eq:thm:lower:extra}
\min\Biggcpar{\frac{1}{p}, \: \frac{p^{k/2}n}{\bigsqpar{k!k\log(1/p)}^{1/(k-1)} }} \cdot (1-p^k)^{n/(k-1)} \; \ge \; C^2 ,
\end{equation}
which together with assumption~\eqref{eq:thm:lower:extra} ensures that~$C n^{\tau} \le s \le n/C$.
To avoid clutter, we shall always tacitly assume that~$n$ and thus~$s \ge n^{\tau}$ are sufficiently large whenever necessary, 
and also treat~$s$ as an integer (the rounding to integers has negligible impact on our calculations).
Since assumption~\eqref{eq:thm:lower:extra} implies\footnote{To see the claimed upper bound on~$k \ge 2$, 
note that~$(1-p^k)^{n/(k-1)} \le 1$ and~\eqref{eq:thm:lower:extra} imply~$1/p \ge C^2 \ge e$, 
which in turn ensures that~\eqref{eq:thm:lower:extra} also implies~$p^{k/2}n \ge C^2 \ge 1$ and thus~$k \le 2\log(n)/\log(1/p)=2\log_{1/p}n$.} that~$k \le 2\log_{1/p}n = o(\sqrt{s})$, 
we may also use~$\tbinom{s}{k} \sim s^k/k!$ and similar asymptotic approximations without further~justifications. 

Turning to the details of our two-step proof approach for Theorem~\ref{thm:lower}, 
we call a $k$-vertex set $\{v_1,\ldots, v_k\} \in \binom{S}{k}$ covered if there is a vertex~$w \in [n]\setminus S$ which is a common neighbor of all the vertices~$v_j$ in~$\Gnp$.
Given a vertex set~$S \subseteq [n]$ of size~$|S|=s$, 
let $X_S$~denote the number of $k$-vertex sets ${\{v_1,\ldots, v_k\} \in \binom{S}{k}}$ which are not covered, 
and let~$Y_S$ denote the number of $k$-vertex cliques~${\{v_1,\ldots, v_k\} \in \binom{S}{k}}$ in~$\Gnp$ whose vertex set is not~covered. 
The first step of our proof approach is to prove that many $k$-vertex sets in~$S$ are not covered, i.e., that~$X_S$ is large; see Lemma~\ref{lem:manypairs}.
With this key concentration result in hand, the second step is to prove that at least one uncovered $k$-vertex set in~$S$ also forms a clique in~$\Gnp$, i.e., that~$Y_S$ is non-zero; see Lemma~\ref{lem:manyedges}.
\begin{lemma}\label{lem:manypairs}
For any~$\delta \in (0,1)$ there is~$C_\delta \ge e$ such that, for any $C \ge C_\delta$, whp the following event~$\fC_\delta$ holds: 
we have ${X_S \ge \delta \E X_S}$ for all vertex sets~$S \subseteq [n]$ of size~$|S|=s$.
\end{lemma}
\begin{lemma}\label{lem:manyedges}
There exists~$C \ge e$ such that whp the following event~$\fS$ holds: 
we have ${Y_S \ge 1}$ for all vertex sets~$S \subseteq [n]$ of size~$|S|=s$.
\end{lemma}
As discussed, the event~$\fS$ from Lemma~\ref{lem:manyedges} implies that~$\chi_c(\Gnp)\ge n/s$, i.e., establishes inequality~\eqref{eq:thm:lower:chi} with room to spare. 
To complete the proof of Theorem~\ref{thm:lower}, it thus remains to give the deferred proofs of Lemmas~\ref{lem:manypairs} and~\ref{lem:manyedges}, which build on top of each other.  
As we shall see, Lemma~\ref{lem:manypairs} is based on a delicate mix of approximation arguments and the bounded differences inequality, 
whereas Lemma~\ref{lem:manyedges} is based on Janson's inequality.

\subsection{\texorpdfstring{Concentration of uncovered $k$-sets: proof of }\ Lemma~\ref{lem:manypairs} }\label{sec:lem:manypairs}
This subsection is devoted to the technical proof of the key concentration result Lemma~\ref{lem:manypairs}. 
With an eye on a union bound argument over all the~$\tbinom{n}{s} \le (ne/s)^s \le p^{-s}$ many~$s$-vertex sets~$S \subseteq [n]$, 
here the natural target bound is~$\Pr(X_S \le \delta \E X_S) \le p^{2s}$, say. 
However, obtaining such good lower tail bounds is more tricky than one might think. 
For example, Janson's inequality for monotone events~\cite{RW} and McDiarmid's bounded differences inequality~\cite{McDiarmid1998} both only lead to insufficient bounds of form~$\Pr(X_S \le \delta \E X_S) \le e^{-O(s)}$, mainly because there are many mild dependencies between the various `not covered' events. 
To go beyond standard concentration inequalities and make the union bound argument~work, 
inspired by~\cite{GW19,WK4,WUT20} we will analyze a carefully designed auxiliary random variable~$X'_S$ that approximates~$X_S$ and has a more tractable tail behavior.

Our approximation arguments hinge on the basic observation that any vertex $w \in [n]\setminus S$ has in expectation only~$|S|p$ neighbors inside~$S$. 
With this in mind, we now introduce the `degree truncation' parameter 
\begin{equation}\label{def:x}
x \: := \: 6 sp ,
\end{equation}
where the definition~\eqref{def:s} of~$s$ implies\footnote{To see the claimed lower bound on~$\max\{np^2,6p^{-(k-2)/2}[\log(1/p)]^{1/(k-1)}\}$, 
note that for~$p \ge n^{-1/3}$ we have~${np^2 \ge n^{1/3} > \log n}$. 
Henceforth assuming~$p \le n^{-1/3}$, in the special case~$k=2$ we have~$6p^{-(k-2)/2}[\log(1/p)]^{1/(k-1)} \ge  6 \cdot \log(n^{1/3}) > \log n$, 
and for~$k \ge 3$ we have $6p^{-(k-2)/2}[\log(1/p)]^{1/(k-1)} \ge n^{(k-2)/6} \ge n^{1/6} > \log n$.} 
that~$x \ge C\max\{np^2,6p^{-(k-2)/2}[\log(1/p)]^{1/(k-1)}\} > C \log n$. 
Similar as for~$s$, we shall treat~$x$ as an integer to avoid clutter (this has negligible impact on our calculations). 
Fixing an $s$-vertex set~$S \subseteq [n]$, for every vertex~$w \in [n] \setminus S$ we then mark the lexicographically first~$x$ edges of~$\Gnp$ in~$\{w\} \times S$ (to clarify: if there are at most $x$ such edges, then we mark all of them). 
We call a pair~$\{u,v\}$ of vertices in~$S$ safely covered if there is a vertex in $w \in [n]\setminus S$ for which $\{u,w\}$ and $\{v,w\}$ are both marked edges. 
We define~$X'_S$~as the number of pairs~$\{u,v\} \in \binom{S}{2}$ which are not safely covered in~$\Gnp$.  
Since every safely covered pair is also covered, we have 
\begin{equation}\label{eq:XSapprox}
X'_S \ge X_S.
\end{equation}
Our strategy for proving Lemma~\ref{lem:manypairs} regarding concentration of~$X_S$ is based on two steps.
First, using degree-counting arguments we shall prove that~$X'_S$ approximates~$X_S$ sufficiently well, see Lemma~\ref{lem:approx}. 
Second, using the bounded differences inequality we shall then prove that~$X'_S$ is concentrated around its expectation, see Lemma~\ref{lem:approx}.  
These two results together suggest~$X_S \approx X'_S \approx \E X'_S \ge \E X_S$, which we shall make precise in the proof of Lemma~\ref{lem:manypairs} below.
\begin{lemma}\label{lem:approx}
For any~$\gamma \in (0,1)$ there is~$C_{0,\gamma} \ge e$ such that, for any~$C \ge C_{0,\gamma}$, whp the following event~$\cA_{\gamma}$ holds: 
we have ${X_S \ge X'_S-\gamma \E X_S}$ for all vertex sets~$S \subseteq [n]$ of size~$|S|=s$.
\end{lemma}
\begin{lemma}\label{lem:conc}
For any~$\gamma\in (0,1)$ there is~$C'_{0,\gamma} \ge e$ such that, for any~$C \ge C'_{0,\gamma}$, whp the following event~$\cC_{\gamma}$ holds: 
we have ${X'_S \ge (1-\gamma) \E X'_S}$ for all vertex sets~$S \subseteq [n]$ of size~$|S|=s$.
\end{lemma}
\begin{proof}[Proof of Lemma~\ref{lem:manypairs} assuming Lemmas~\ref{lem:approx}--\ref{lem:conc}]
Set~$\gamma := (1-\delta)/2$. 
For any vertex set~$S \subseteq [n]$ of size~$|S|=s$, 
note that the event $\cA_{\gamma} \cap \cC_{\gamma}$ and inequality~$X'_S \ge X_S$ together imply that 
\[
X_S \ge X'_S-\gamma \E X_S \ge (1-\gamma) \E X'_S -\gamma \E X_S 
\ge (1-2\gamma) \E X_S = \delta \E X_S.
\]
This completes the proof of Lemma~\ref{lem:manypairs} with~$C_\delta := \max\{C_{0,\gamma},C'_{0,\gamma}\}$, assuming Lemmas~\ref{lem:approx}--\ref{lem:conc}. 
\end{proof}

To complete the proof of Lemma~\ref{lem:manypairs}, it remains to give the proofs of Lemmas~\ref{lem:approx}--\ref{lem:conc}. 
We start with the concentration result Lemma~\ref{lem:conc}. 
Here the key point will be that changing the status of a single edge can not alter~$X'_S$ substantially `by~definition', 
which allows us to show concentration of~$X'_S$ via the following 
version of the bounded differences inequality for Bernoulli variables, 
see~{\cite[Corollary~1.4]{War}} and~{\cite[Theorem~3.8]{McDiarmid1998}}. 
\begin{lemma}[Bounded differences inequality]\label{typical BDI}
Let $\psi = (\xi_1,\dots , \xi_N)$ be a family of independent random variables with~$\xi_k\in \{0, 1\}$ and~$\Pr(\xi_k = 1) \le p$ for all~$k \in [N]$. 
Let $f:\{0,1\}^{N} \to \RR$ be a function, and assume that there exists~$D>0$ such that $|f(x) - f(y)|\le D$ for any two binary vectors $x,y \in \{0,1\}^N$ that differ in at most one coordinate. 
Setting~$Z:=f(\psi)$, for all~$t \ge 0$ we have  
\begin{equation}\label{eq:BDI}
    \Pr(Z\le \E Z - t) \: \le \: \exp\biggpar{-\frac{t^2}{2(NpD^2+Dt)}}.
\end{equation}
\end{lemma}
\begin{proof}[Proof of Lemma~\ref{lem:conc}]
Fix a vertex set~$S \subseteq [n]$ of size~$|S|=s$. 
Note that $X'_S$ is a deterministic function~$f$ of the $N:=s(n-s)$ independent Bernoulli random variables $(\xi_1,\dots , \xi_N):=(\indic{uv\in E(\Gnp)})_{u\in S, v\in V\setminus S}$, 
which encode the edge-status of the corresponding vertex pairs in~$\Gnp$. 
Recalling the definition of marked edges right after~\eqref{def:x}, the crux is that any vertex~$w \in V\setminus S$ is always adjacent to at most~$x$ marked edges. 
It follows that adding or deleting any edge~$\{w,u\}$ with~$w \in V\setminus S$ and $u \in S$ can only change (increase or decrease) the number of safely covered pairs by at most~$\binom{x}{k-1}$. 
Hence we can apply Lemma~\ref{typical BDI} to~$Z:=X'_S$ with~$N:=s(n-s) \le sn$ and~$D:=\binom{x}{k-1}$, 
so that inequality~\eqref{eq:BDI} with~$t:= \gamma\E X'_S$~gives 
\begin{equation}\label{eq:BDI:X'S}
 \Pr(X'_S \le (1-\gamma) \E X'_S) \: \le \: \exp\biggpar{-\min\biggcpar{\frac{\gamma^2(\E X'_S)^2}{4snp\binom{x}{k-1}^2}, \: \frac{\gamma\E X'_S}{4\binom{x}{k-1}}}} .
\end{equation}

Gearing up towards a union bound argument, we now estimate the exponents in inequality~\eqref{eq:BDI:X'S}. 
Using inequality~$X'_S \ge X_S$ from~\eqref{eq:XSapprox}, we first obtain the auxiliary estimate  
\begin{equation}\label{eq:ExpXs}
\E X'_S \ge \E X_S \ge \tbinom{s}{k}(1-p^k)^{n-s} \ge (1-o(1)) s^k(1-p^k)^{n}/k!.
\end{equation} 
Together with $\tbinom{x}{k-1} \le x^{k-1}/(k-1)!$ and~$x/s=6p$, we then infer that 
\begin{equation}\label{eq:BDI:X'S:exp1}
\Pi := \min\biggcpar{\frac{\gamma^2(\E X'_S)^2}{4snp\binom{x}{k-1}^2}, \: \frac{\gamma\E X'_S}{4\binom{x}{k-1}}} 
\ge (1-o(1)) s \cdot \min\biggcpar{\biggpar{\frac{\gamma (1-p^k)^{n}}{2 k \sqrt{np}(6p)^{k-1}}}^2, \: \frac{\gamma(1-p^k)^{n}}{4k(6p)^{k-1}}} .
\end{equation}
Using assumption~\eqref{eq:thm:lower:p:upper:0} to bound~$(6p)^{k-1}$ from above, 
for~$C \ge 4/\gamma$ it follows in view of~$np \ge n^{\tau}$ that 
\begin{equation}\label{eq:BDI:X'S:exp2}
\Pi \ge \frac{s}{2} \cdot \min\biggcpar{\biggpar{\frac{C^{k-1}\gamma}{2}}^2 \cdot \log(1/p), \: \frac{C^{k-1}\gamma }{4} \cdot \sqrt{np \log(1/p)}} \ge 2 s \log(1/p). 
\end{equation}

Finally, by inserting~\eqref{eq:BDI:X'S:exp1}--\eqref{eq:BDI:X'S:exp2} into the tail bound~\eqref{eq:BDI:X'S}, 
using~$\tbinom{n}{s} \le (ne/s)^s \le p^{-s}$ it follows that 
\begin{equation*}
\Pr(\cC_{\gamma}) \le \sum_{S \subseteq [n]: |S|=s} \Pr(X'_S\ge (1-\gamma) \E X'_S)  \le \tbinom{n}{s} \cdot p^{-2s} \le p^{-s} = o(1) ,
\end{equation*}
completing the proof of Lemma~\ref{lem:conc} with~$C'_{0,\gamma} := 4/\gamma$. 
\end{proof}

We now give the proof of the approximation result Lemma~\ref{lem:approx}. 
Here the basic strategy will be to bound the difference $X'_S - X_S$ via (slightly tedious) degree counting~arguments that are inspired by~{\cite{ER61,SW20}}. 
\begin{proof}[Proof of Lemma~\ref{lem:approx}]
Fixing a vertex set~$S \subseteq [n]$ of size~$|S|=s$, recall that unmarked edges are ignored in the definition of~$X'_S$ given right before~\eqref{eq:XSapprox}, 
and that each of these unmarked edges must contain a vertex $w \in [n]\setminus S$ with at least~$x$ neighbors in~$S$. 
To get a handle on how much these unmarked edges contribute to the approximation error~${X_S-X'_S}$, 
we thus group the vertices outside of~$S$ according to their number of neighbors in~$S$. 
More precisely, we define~$V_{i,S}$ as the set of vertices $w \in [n]\setminus S$ that have between $2^i x$ and $2^{i+1} x$ neighbors in~$S$. 
Since the relevant edges adjacent to different vertices $w \in V \setminus S$ are all distinct, 
by standard random graph estimates we see that 
\begin{equation*}
\Pr(|V_{i,S}| \ge z) \le \binom{n}{z} \binom{z s}{z 2^i x } p^{z 2^i x} \le \biggsqpar{ n \Bigpar{\frac{spe}{2^i x}}^{2^i x}}^{z} \qquad \text{for any integer~$z \ge 1$.}
\end{equation*}
With foresight, for any integer~$i \ge 0$ we now introduce the `vertex counting' parameter 
\begin{equation*}
z_i \: := \: \frac{12 s \log_2(en/s)}{(i+1)2^{i+1}x} ,
\end{equation*}
where~$s \le n$ implies that~$z_i>0$. 
Recalling that~$x =6sp \ge C\log n$, as established below~\eqref{def:x}, for~$C \ge 2/\log 2$ we infer that~$x/2 \ge \log_2 n$ and 
\[
\Pr(|V_{i,S}| \ge \ceil{z_i}) \: \le \: \Bigsqpar{ n \cdot 2^{-(i+1)2^i x}}^{\ceil{z_i}} \le 2^{-(i+1)2^{i-1} x z_i} 
= (en/s)^{-3s}  .
\]
Let~$\cB$ denote the event that, for all $s$-vertex set~$S \subseteq [n]$, we have $|V_{i,S}| \le z_i$ for all integers~$i \ge 0$ with~${2^ix \le 2np}$.   
Using~$en/s \ge e$ and~$s \ge n^{\tau}$, it follows that 
\begin{equation}\label{eq:event:B}
\begin{split}
\Pr(\neg \cB) \le \sum_{S \subseteq [n]: |S|=s} \sum_{i \ge 0: 2^i x \le 2np} \Pr(|V_{i,S}| \ge \ceil{z_i}) \le \tbinom{n}{s} \cdot O(\log n) \cdot (en/s)^{-3s} \le (en/s)^{-s}  = o(1). 
\end{split}
\end{equation}
To eventually deduce the bound~$|V_{i,S}|=0$ for the remaining integers~$i \ge 0$ with~${2^ix > 2np}$, 
let~$\cD$ denote the auxiliary event that the maximum degree of~$\Gnp$ is at most~$2np$. 
Using standard Chernoff bounds (such as~\mbox{\cite[Corollary 2.3]{JLR}}) and assumption~\eqref{eq:thm:lower:p:upper:0} we infer~that
\begin{equation}\label{eq:chernoff:degree}
\Pr(\neg\cD) \le n \cdot e^{-\Theta(np)} \le e^{-\Omega(n^{\tau})} = o(1),
\end{equation}
which together with~\eqref{eq:event:B} establishes that the event~$\cB \cap \cD$ holds whp.

It remains to show that the event~$\cB \cap \cD$ implies the approximation event~$\cA_\gamma$,  
exploiting the implied size bounds~${|V_{i,S}| \le z_i \indic{2^i x \le 2np}}$ for all~$i \ge 0$. 
Since every unmarked edge contains a vertex from one of the above-defined vertex sets~$V_{i,S}$, 
using~$\tbinom{2^{i+1}x}{k} \le (2^{i+1}x)^k/k!$ it follows that 
\[
X'_S - X_S \le \sum_{i \ge 0} |V_{i,S}| \tbinom{2^{i+1}x}{k} 
\le \frac{12s}{k!} \sum_{i \ge 0: 2^{i}x \le 2np} \frac{(2^{i+1}x)^{k-1}  \log_2(en/s)}{i+1}.
\]
The above sum is intuitively dominated by the last term (up to constant factors), 
and to make this rigorous we shall now analyze the behavior of the term~$\log_2(en/s)/(i+1)$. 
In particular, if $2^{i+1}x \ge enp/\log(en/s) =: \Lambda$ holds, then using~$x=6sp$ and~$en/s \ge C$ we infer for all sufficiently large~$C \ge C_0$ that 
\[
i+1 \ge \log_2\Bigpar{6^{-1} \cdot (en/s)/\log(en/s)} \ge \tfrac{1}{2} \log_2(en/s) .
\]
By distinguishing whether $2^{i+1}x$ is smaller or larger than~$\Lambda=enp/\log(en/s)$, we thus arrive at 
\[
X'_S - X_S 
\le \frac{12s(2x)^{k-1}}{k!} \biggsqpar{\log_2(en/s)\sum_{i \ge 0: 2^{i}x < \Lambda/2} 2^{i(k-1)}  + 2\sum_{i \ge 0: \Lambda/2 \le 2^{i}x \le 2np} 2^{i(k-1)}}.
\]
For~$\Pi \in \{\Lambda/2,2np\}$ we have~$(2x)^{k-1}\sum_{i \ge 0: 2^i x \le \Pi} 2^{i(k-1)} \le 2 (2\Pi)^{k-1}$,  
so it follows that
\[
X'_S - X_S \le \frac{12s}{k!} \Bigsqpar{2(e np)^{k-1} + 4(4np)^{k-1}} 
\le \frac{72}{k!} s (4np)^{k-1} .
\]
Since~$\E X_S \ge (1-o(1)) s^k(1-p^k)^n/k!$ by~\eqref{eq:ExpXs}, 
using the definition~\eqref{def:s} of~$s$ it follows for~$C \ge 292/\gamma$ that 
\begin{equation*}
\frac{X'_S-X_S}{\E X_S} \le \frac{73(4np/s)^{k-1}}{(1-p^k)^{n}} \le 73 (4/C)^{k-1} \le \gamma,  
\end{equation*}
completing the proof of Lemma~\ref{lem:approx} with~$C_{0,\gamma}:=\max\{292/\gamma,C_0\}$.    
\end{proof}

\subsection{\texorpdfstring{Existence of uncovered $k$-cliques: proof of }\ Lemma~\ref{lem:manyedges}}\label{sec:lem:manyedges}
This subsection is devoted to the remaining proof of Lemma~\ref{lem:manyedges}. 
Here our starting point is the event~$\fC_\delta$ from Lemma~\ref{lem:manypairs}, which guarantees existence of many uncovered $k$-vertex sets in $S$. 
Using Janson's inequality, we will then prove that typically at least one of these uncovered $k$-vertex sets is in fact a clique. 
\begin{proof}[Proof of Lemma~\ref{lem:manyedges}]
Set~$\delta:=1/2$, and fix a vertex set~$S \subseteq [n]$ of size~$|S|=s$.   
We henceforth condition on the random variable
\begin{equation}\label{def:xiS}
\Xi_S \: := \: \bigpar{\indic{uv\in E(\Gnp)}}_{uv \in \binom{[n]}{2} \setminus \binom{S}{2}}, 
\end{equation}
which encodes the edge-status of all vertex pairs of~$\Gnp$ except for those inside~$S$. 
Since the number~$X_S$ of uncovered $k$-vertex sets~$\{v_1, \ldots, v_k\} \in \binom{S}{k}$ is determined by~$\Xi$, 
it follows that 
\begin{equation}\label{eq:Ys:bound:0}
\Pr(Y_S=0 \text{ and } X_S \ge \delta \E X_S) = \E\bigpar{\Pr(Y_S=0 \mid \Xi_S)\indic{X_S \ge \delta \E X_S} } .
\end{equation}
Recall that~$Y_S$ counts the number of uncovered $k$-vertex sets~$\{v_1, \ldots, v_k\} \in \binom{S}{k}$ which form a clique in~$\Gnp$. 
After conditioning on~$\Xi_S$, the crux is that all potential edges inside~$S$ are still included independently with probability~$p$. 
By applying Janson's inequality (see, e.g., \mbox{\cite[Theorem~2.14]{JLR}}) 
to~$Y_S$ it thus follows that 
\begin{equation}\label{eq:Ys:bound}
\Pr(Y_S=0 \mid \Xi_S) \le  \exp\biggpar{-\frac{\mu^2}{2(\mu+\Delta)}} ,
\end{equation}
where the associated expectation and correlation parameters~$\mu$ and $\Delta$ satisfy 
\begin{equation}\label{eq:Ys:bound:parameters}
\mu := \E(Y_S \mid \Xi_S) = X_S p^{\binom{k}{2}} \quad \text{ and } \quad \Delta \le \sum_{2 \le i < k} X_S \binom{k}{i}\binom{s-k}{k-i}p^{2\binom{k}{2}-\binom{i}{2}} 
\end{equation}
by standard random graph estimates. 
It will be convenient to~write 
\begin{equation}\label{eq:Ys:bound:parameters2}
\frac{\mu+\Delta}{\mu} \le \sum_{2 \le i \le k} \underbrace{\binom{k}{i}\binom{s-k}{k-i}p^{\binom{k}{2}-\binom{i}{2}}}_{=:a_i} = \sum_{2 \le i \le k} a_i .
\end{equation}
Considering ratios of consecutive terms, using $k = o(\sqrt{s})$ we see that for $2 \le i \le k-1$ we have 
\[
\frac{a_{i+1}}{a_i} = \frac{(k-i)^2}{(i+1) (s-2k+i+1)p^i} \sim \frac{(k-i)^2}{(i+1) sp^i} .
\]
Note that if we increase the value of the integer~$2 \le i \le k-2$ by one, 
then~$(k-i)^2/(i+1)$ changes by at most a constant~factor (that is bounded from above and below by universal constants that do not depend on~$C$), 
whereas~$sp^i$ changes by a factor of~$p$.
Since assumption~\eqref{eq:thm:lower:extra} implies~$p \le 1/C$, for all sufficiently large~$C \ge C'_0$ it follows the ratio~$a_{i+1}/a_{i}$ is increasing with~$i$, 
which in turn implies that the maximum of~${a_2, \ldots, a_k}$ is either~$a_2$ or~$a_k$. 
It follows~that the exponent in inequality~\eqref{eq:Ys:bound} satisfies 
\begin{equation}\label{eq:Ys:exp}
\frac{\mu^2}{2(\mu+\Delta)} \ge \frac{\mu}{2\sum_{2 \le i \le k} a_i} \ge \frac{X_S p^{\binom{k}{2}}}{2k\max\{a_2,a_k\}} 
\ge \frac{\delta \E X_S p^{\binom{k}{2}}}{2k\max\{a_2,a_k\}} \indic{X_S \ge \delta \E X_S} .
\end{equation}
Using first~$a_k=1$ together with the estimate~$\E X_S \ge (1-o(1)) s^k(1-p^k)^n/k!$ from~\eqref{eq:ExpXs}, 
and subsequently the definition~\eqref{def:s} of~$s$, 
for~$C \ge 8/\delta$ it follows that 
\begin{equation}\label{eq:Ys:exp:1}
\frac{\delta \E X_S p^{\binom{k}{2}}}{2ka_k} \ge \frac{\delta s}{4} \cdot \frac{s^{k-1}p^{\binom{k}{2}} (1-p^k)^{n}}{k!k} \ge \frac{\delta s}{4} \cdot C^{k-1}\log(1/p) \ge 2s \log(1/p) .
\end{equation}
When estimating inequality~\eqref{eq:Ys:exp:2} below, 
note that it suffices to consider the case~$k \ge 3$ (since~\eqref{eq:Ys:exp:1} already covers the case~$k=2$).
Using first~$a_2\le k^2 s^{k-2} p^{\binom{k}{2}-1}/(k-2)!$ together with the estimate~\eqref{eq:ExpXs} for~$\E X_S$,  
next the definition~\eqref{def:s} of~$s$, 
and subsequently assumption~\eqref{eq:thm:lower:p:upper} together with~$k \ge 3$, 
for~$C \ge 8/\delta$ it follows that 
\begin{equation}\label{eq:Ys:exp:2}
\begin{split}
\frac{\delta \E X_S p^{\binom{k}{2}}}{2k a_2} 
\ge \frac{\delta s}{4} \cdot \frac{s p (1-p^k)^n}{k^5} & \ge \frac{\delta C s}{4} \cdot \frac{\max\bigcpar{np^2, \: p^{-(k-2)/2}}(1-p^k)^{n (k-2)/(k-1)}}{k^5} \\
& \ge \frac{\delta C s}{4} \cdot \log(1/p) \ge 2s \log(1/p) .
\end{split}
\end{equation}
Since~$X_S$ is determined by the random variable~$\Xi_S$, 
by inserting~\eqref{eq:Ys:exp}--\eqref{eq:Ys:exp:2} into~\eqref{eq:Ys:bound}  
it follows that 
\begin{equation}\label{eq:Ys:bound:2}
\Pr(Y_S=0 \mid \Xi_S)\indic{X_S \ge \delta \E X_S} \le  \exp\bigpar{-2s \log(1/p)} = p^{2s} .
\end{equation}

Finally, since the event~$\fC_\delta$ from Lemma~\ref{lem:manypairs} implies~$X_S \ge \delta \E X_S$ for all $s$-vertex sets~$S \subseteq [n]$, 
using the probabilistic bounds~\eqref{eq:Ys:bound:0} and~\eqref{eq:Ys:bound:2} together with~$\tbinom{n}{s} \le (ne/s)^s \le p^{-s}$ it follows that 
\begin{equation*}
\Pr(\neg \fS \text{ and } \fC_\delta) \le \sum_{S \subseteq [n]: |S|=s} \Pr(Y_S=0 \text{ and } X_S \ge \delta \E X_S) \le \tbinom{n}{s} \cdot p^{2s} \le p^{s} = o(1) .
\end{equation*}
This completes the proof of Lemma~\ref{lem:manyedges} with~$C:=\max\{8/\delta,C'_0,C_\delta\}$, since~$\Pr(\neg\fC_\delta) = o(1)$ by Lemma~\ref{lem:manypairs}. 
\end{proof}

\section{Upper bound: proof of Theorem~\ref{thm:upper}}\label{sec: upper bounds} %
In this section we prove our upper bound result Theorem~\ref{thm:upper}.
Our proof strategy hinges on a partition of the vertices of the random graph~$\Gnp$ 
into ${O\bigl(p^{3/2}n/\sqrt{\log n}\bigr)}$ many disjoint vertex sets, 
where we shall use different sets of colors for each vertex set.
By bounding certain maximum degrees inside each such vertex set, 
we are then eventually able to obtain a valid clique coloring of~$\Gnp$ 
with the number of colors claimed by Theorem~\ref{thm:upper}.

\subsection{Part~\eqref{pt 1'} of Theorem~\ref{thm:upper}}\label{sec:upper:part1}
In this subsection we prove part~\eqref{pt 1'} of Theorem~\ref{thm:upper}.
To this end we fix a partition $S_1 \cupdot \cdots \cupdot S_r=[n]$ of the vertices of~$\Gnp$ into 
\begin{equation}\label{def:r}
r := \Big\lceil p^{3/2} n / \sqrt{\log n}\Big\rceil
\end{equation}
sets of almost equal size (i.e., whose sizes differ by at most one).
For each $i \in [r]$, we now construct a subgraph~$G_i \subseteq \Gnp[S_i]$ by deleting edges from $\Gnp[S_i]$ 
which (a)~are contained in a triangle of~$\Gnp$, but at the same time (b)~are not contained in a triangle of~$\Gnp[S_i]$. 
The crux is that every inclusion-maximal clique of~$\Gnp$ that only contains vertices from one set~$S_i$ 
also forms an inclusion-maximal clique in~$G_i \subseteq \Gnp[S_i]$.
In order to find a valid clique coloring of~$\Gnp$, by using different colors for each set~$S_i$ it thus suffices to find a valid clique coloring for each subgraph~$G_i$. 
The following convenient result of Joret, Micek, Reed and Smid~\cite{JMRS} reduces the aforementioned coloring problem to the maximum degree of the~$G_i$. 
\begin{theorem}\label{thm:JMRS}\cite{JMRS}
For every~$\varepsilon > 0$ there is~$\Delta_{\varepsilon}>0$ such that every graph~$G$ with maximum degree~$\Delta \ge \Delta_{\varepsilon}$ has clique chromatic number at most~$(1+\varepsilon)\Delta/\log \Delta$.
\end{theorem}
\noindent
Coloring each~$G_i$ using Theorem~\ref{thm:JMRS} with~$\varepsilon=1$, 
by invoking the maximum degree result Lemma~\ref{lem Delta} below 
we then obtain a valid clique coloring of~$\Gnp$ that whp uses at most
\[
r \cdot 84\Gamma/\log (42\Gamma) 
\; \le \; O(1) \cdot \max\biggcpar{\frac{e^{-np^2}np}{\log n}, \: \frac{p^{3/2}n}{\sqrt{\log n}}} \cdot \frac{\log n}{\log \Gamma}
\]
many colors.
To complete the proof of part~\eqref{pt 1'} of Theorem~\ref{thm:upper}, it thus remains to prove Lemma~\ref{lem Delta}. 
\begin{lemma}\label{lem Delta}
Let~$\Gamma := \max\bigl\{e^{-np^2}np/r, \: \log n\bigr\}$. 
Then, whp, the maximum degree of~$G_i$ is at most~$42 \Gamma$ for all~$i\in [r]$. 
\end{lemma}
\begin{proof}
We henceforth fix a partition number~$i \in [r]$ and a vertex~$v \in S_i$. 
In view of edge condition~(b) from the~$G_i \subseteq \Gnp[S_i]$ construction, 
we first bound the number~$X_v$ of triangles of~$\Gnp[S_i]$ that contain~$v$. 
We denote by~$N_i(v)$ the set of vertices from~$S_i$ which are adjacent to~$v$ in~$\Gnp$. 
Note that~$\E |N_i(v)| \le \ceil{n/r} \cdot p = \Theta(p^{-1/2}\sqrt{\log n}) = \omega(\log n)$. 
Let $\cU_i := {\{U \subseteq S_i: |U| \le 2np/r\}}$, and define~$Z_U$ as the number of edges in~$\Gnp[U]$. 
Using standard Chernoff bounds (such as~\mbox{\cite[Corollary 2.3]{JLR}}) and the independence of edges, it follows that 
\begin{equation}\label{eq:Xv}
\begin{split}
\Pr(X_v \ge 12 \log n) &\le \Pr(|N_i(v)| \ge 2np/r) + \sum_{U \in \cU_i}\underbrace{\Pr(X_v \ge 12 \log n \: | \: N_i(v)=U)}_{=\Pr(Z_U \ge 12 \log n)}\Pr(N_i(v)=U)\\
& \le o(n^{-2}) + \max_{U \in \cU_i} \Pr(Z_U \ge 12 \log n). 
\end{split}
\end{equation}
Note that for all~$U \in \cU_i$ we have $\E Z_U = \binom{|U|}{2}p \le 2n^2p^3/r^2 \le 2 \log n$. 
Hence 
\begin{equation}\label{eq:ZU}
\Pr(Z_U \ge 12 \log n) 
\le \binom{\binom{|U|}{2}}{\ceil{12 \log n}} p^{\ceil{12 \log n}} 
\le \biggpar{\frac{e\E Z_U }{12 \log n}}^{\ceil{12 \log n}} 
= o(n^{-2}). 
\end{equation}

With an eye on the edge condition~(a) from the~$G_i \subseteq \Gnp[S_i]$ construction, 
we next bound the number~$Y_v$ of vertices~$u \in N_i(v)$ for which there is no vertex~$w \in [n] \setminus S_i$ that is adjacent to both~$u$ and~$v$. 
We denote by~$N^+_i(v)$ the set of vertices from~$[n] \setminus S_i$ which are adjacent to~$v$ in~$\Gnp$. 
In view of~$\E |N^+_i(v)| = (n-|S_i|)p$, we define~$\cT_i := {\{T \subseteq [n]\setminus S_i: |T| \ge (n-|S_i|)p - \sqrt{(n-|S_i|) p} \log n\}}$. 
Proceeding similarly to~\eqref{eq:Xv}, using standard Chernoff bounds it follows that 
\begin{equation}\label{eq:Yv}
\begin{split}
\Pr(Y_v \ge 18 \Gamma) &\le o(n^{-2}) + \max_{U \in \cU_i,T \in \cT_i} \underbrace{\Pr(Y_v \ge 18 \Gamma \: | \: N_i(v)=U, \: N^+_i(v)=T)}_{=\Pr(Z_{U,T} \ge 18\Gamma)}, 
\end{split}
\end{equation}
where~$Z_{U,T}$ is a binomial random variable with~$|U|$ independent trials and success probability~$(1-p)^{|T|}$. 
Note that for all~$U \in \cU_i$ and~$T \in \cT_i$ we have 
\begin{equation}\label{eq:EZUT}
\E Z_{U,T} = |U|(1-p)^{|T|} 
\le 2np/r \cdot e^{-np^2 + |S_i|p^2 + p\sqrt{np}\log n} 
\le 3np/r \cdot e^{-np^2} ,
\end{equation}
since~$|S_i|p^2 =\ceil{n/r}p^2= O(\sqrt{p\log n})=o(1)$ and~$p\sqrt{np}\log n = \sqrt{np^3 (\log n)^2} = o(1)$. 
In view of~$\E Z_{U,T} \le 3 \Gamma$ and~$\Gamma \ge \log n$, it follows similarly to~\eqref{eq:ZU} that 
\begin{equation}\label{eq:ZUT}
\Pr(Z_{U,T} \ge 18 \Gamma) 
\le \biggpar{\frac{e\E Z_{U,T} }{18 \Gamma}}^{\ceil{18 \Gamma}} 
= o(n^{-2}). 
\end{equation}

Finally, taking a union bound over all~$i \in [r]$ and~$v \in S_i$, using inequalities~\mbox{\eqref{eq:Xv}--\eqref{eq:ZU}} and~\mbox{\eqref{eq:Yv}--\eqref{eq:ZUT}} 
it routinely follows that whp~$2X_v+ Y_v \le 24 \log n + 18 \Gamma \le 42 \Gamma$ for all~$i \in [r]$ and~$v \in S_i$. 
This completes the proof of \refL{lem Delta}, since the edge conditions~(a) and~(b) from the construction of~$G_i \subseteq \Gnp[S_i]$ 
ensure that for each vertex~$v \in S_i$ the degree inside~$G_i$ is at most~$2X_v+ Y_v$. 
\end{proof}

\subsection{Part~\eqref{pt 3'} of Theorem~\ref{thm:upper}}\label{sec:upper:part2}
In this subsection we prove part~\eqref{pt 3'} of Theorem~\ref{thm:upper}. 
Setting $\phat := 2p$, there is a natural coupling of~$\Gnp$ and~$\Gnph$ that satisfies~$\Gnp \subseteq \Gnph$. 
Since~$\phat = \Theta(p)$ and $\phat > \sqrt{(2\log n)/n}$, 
by the proof of Theorem~3.1 in~\cite{MMP} we know that~$\Gnph$ whp admits a valid clique coloring using at most 
\begin{equation}\label{def:rhat}
\hat{r} = O\Bigl(p^{3/2}n/\sqrt{\log n}\Bigr)
\end{equation}
colors, 
where each color class $S_1, \ldots , S_{\hat{r}}$ is triangle-free. 
For each~$i \in [\hat{r}]$, we now construct a subgraph~${G_i \subseteq \Gnp[S_i]}$ as in Section~\ref{sec:upper:part1}. 
Using~$\Gnp \subseteq \Gnph$ we infer that each~$\Gnp[S_i]$ is triangle-free, 
so from \refL{lem final} below it follows whp that each~$G_i$ has maximum degree at most~$\lambda = \ceil{9(\log n)/\xi}$, 
which in turn implies that each~$G_i$ admits a valid vertex coloring (and thus clique coloring) using at most~$\lambda+1 \le 2\lambda$ colors. 
By using different colors for each set~$S_i$,
analogously to Section~\ref{sec:upper:part1} we then obtain a valid clique coloring of~$\Gnp$ that whp uses~at~most
\begin{equation*}
\hat{r} \cdot 2\lambda 
\; \le \; O(1) \cdot \frac{p^{3/2}n}{\sqrt{\log n}} \cdot \frac{\log n}{\xi}
\end{equation*}
many colors. 
To complete the proof of part~\eqref{pt 3'} of Theorem~\ref{thm:upper}, it thus remains to prove Lemma~\ref{lem final}. 
\begin{lemma}\label{lem final}
The following event~$\cG$ holds whp: in~$\Gnp$ every vertex is contained in at most~${\lambda:= \ceil{9(\log n)/\xi}}$ many edges that are not contained in triangles of~$\Gnp$.
\end{lemma}
\begin{proof}
We denote by~$N(v)$ the set of vertices which are adjacent to~$v$ in~$\Gnp$,  and define~$Z_U$ as the number of isolated vertices in~$\Gnp[U]$. 
Using a union bound argument, we infer that
\begin{equation}\label{eq:cG}
\Pr(\neg \cG) \le \sum_{v \in [n]} \Pr(Z_{N(v)} \ge \lambda) .
\end{equation}
In view of~$\E N(v) =(n-1)p$, we define~$\cU := {\{U \subseteq [n]: ||U| -np| \le \sqrt{np}\log n\}}$. 
Proceeding similarly to~\eqref{eq:Xv}, using standard Chernoff bounds and the independence of edges it follows that 
\begin{equation}\label{eq:ZNv}
\begin{split}
\Pr(Z_{N(v)}\ge \lambda) 
& \le o(n^{-1}) + \max_{U \in \cU} \Pr(Z_{U}\ge \lambda). 
\end{split}
\end{equation}
Similar to~\eqref{eq:ZU} and~\eqref{eq:EZUT}, using standard random graph estimates and~$np < 2 \sqrt{n \log n}$ we obtain that 
\begin{equation*}
\Pr(Z_{U}\ge \lambda)
\le |U|^\lambda (1-p)^{\lambda(|U|-\lambda) + \binom{\lambda}{2}} 
\le \Bigpar{2np \cdot e^{-np^2+p\sqrt{np}\log n + p\lambda }}^{\lambda} 
\le \Bigpar{5 \sqrt{n \log n} \cdot e^{-np^2}}^{\lambda} ,
\end{equation*}
since~$p\sqrt{np}\log n = \sqrt{np^3(\log n)^2} = o(1)$ and~$p \lambda = O(p \log n) = o(1)$. 
Noting~${np^2}={\xi/2 + (\log n + \log \log n)/2}$ and~${\xi \ge \log \log n \gg 1}$, 
it follows by definition of~$\lambda \ge 9(\log n)/\xi$ that 
\begin{equation*}
\Pr(Z_{U}\ge \lambda) 
\le \Bigpar{5 e^{-\xi/2}}^{\lambda} \le e^{-\lambda\xi/3} 
= o(n^{-1}) ,
\end{equation*}
which together with~\eqref{eq:cG}--\eqref{eq:ZNv} establishes that~$\cG$ holds whp, completing the proof of \refL{lem final}. 
\end{proof}

\section{Main contributions: proof of Theorem~\ref{thm:main1} and~\ref{thm:main2}}\label{sec:technical}
In this section we prove Theorem~\ref{thm:main1} and~\ref{thm:main2}  
by several case distinctions involving 
\begin{equation}\label{eq:points}
p_0 := \frac{\log n}{n}, \qquad  p^*_0 := n^{-0.6}, \qquad p_1 := \sqrt{\frac{\log n - 3\log \log n}{4n}} \quad \text{ and } \quad p_2 := \biggpar{\frac{\log n}{n^2}}^{1/5} .
\end{equation}

\subsection{Proof of Theorem~\ref{thm:main1}}
Gearing up towards the proof of Theorem~\ref{thm:main1}, it is routine to check that the maximum in~\eqref{eq:main1} satisfies
\begin{equation}\label{eq:max1}
\max\biggcpar{\frac{e^{-np^2}np}{\log(np)}, \: \frac{p^{3/2}n}{\sqrt{\log n}}}
= \begin{cases}
\Theta\Bigpar{\frac{e^{-np^2}np}{\log n}} & \text{if $p \in [p^*_0,p_1]$,} \\
\Theta\Bigpar{\frac{p^{3/2}n}{\sqrt{\log n}}} & \text{if $p \in [p_1,1]$.} \\
\end{cases}
\end{equation}
With foresight, we shall now prove~\eqref{eq:main1} of Theorem~\ref{thm:main1} for the larger range~$p \in [p^*_0, p_2]$. 
\begin{proof}[Proof of Theorem~\ref{thm:main1}]
We start with the upper bound in~\eqref{eq:main1} for~$p \in [p^*_0, p_2]$, 
which with an extra factor of~$(\log n)/(\log \Gamma) = o(\log n)$  immediately follows from part~\eqref{pt 1'} of Theorem~\ref{thm:upper}.
Furthermore, for~$4np^2/(\log n) \le 1-\Omega(1)$ it is routine to check that  $(\log n)/(\log \Gamma) = O(1)$, 
which implies that in~\eqref{eq:main1} no extra polylogarithmic factors are needed.
To argue similarly for~$4np^2/(\log n) \ge 2+\Omega(1)$, 
first note that by~\eqref{eq:max1} the maximum in~\eqref{eq:main1} is of order~$p^{3/2}n/\sqrt{\log n}$, 
and also that~$(\log n)/\xi = O(1)$ in~\eqref{eq:thm:upper:chi:2}.   
Part~\eqref{pt 3'} of Theorem~\ref{thm:upper} thus shows that in~\eqref{eq:main1} no extra polylogarithmic factors are needed for~$2 +\Omega(1) \le 4np^2/(\log n) \le 12$, 
and for~$8 \le 4np^2/(\log n)$ this similarly follows from~\cite[Theorem~3.1]{MMP}, 
establishing the upper bound in~\eqref{eq:main1}.

We now turn to the more involved lower bound in~\eqref{eq:main1} for~$p^*_0 \le p \le p_2$. 
To this end we define 
\begin{equation}\label{eq:points:1}
p^*_1 := \sqrt{\frac{\log n - 5\log \log n-4\log(24C)}{4n}},
\end{equation}
where the constant~$C \ge e$ is from Theorem~\ref{thm:lower}. 
For~$p \in [p^*_0,p^*_1]$ it follows that~${p^4n = o(1)}$ and ${\log(1/p) \cdot np^3} \le {(\log n) \cdot n^{-1/2} (\log n)^{3/2}}$ as well as~${4np^2} \le {\log n - 5 \log \log n - 4 \log(24C)}$ together imply that
\begin{equation}
\frac{(1-p^2)^n}{6C\sqrt{2^2 \log(1/p)np} \cdot p} 
\: \ge \: 
\frac{(1-o(1))e^{-np^2} \cdot n^{1/4}}{12 C (\log n)^{5/4}} \ge 2-o(1) > 1 ,
\end{equation}
which verifies assumptions~\eqref{eq:thm:lower:p:upper:0}--\eqref{eq:thm:lower:p:upper} of Theorem~\ref{thm:lower} for~$k=2$ and~$\tau=0.4$.  
Since~$p^2n/\log (1/p) = O(1)$, by invoking~\eqref{eq:thm:lower:chi} with~$k=2$ it follows that there is a constant~$c_1>0$ such that~whp
\begin{equation}\label{eq:lower:1}
\chi_c\bigpar{\Gnp} \: \ge \: c_1 \cdot \frac{np}{\log n} \cdot e^{-np^2} \qquad\qquad \text{if~$p \in [p^*_0,p^*_1]$.}
\end{equation}
For~$p \in [p^*_0,p_2]$ it routinely follows that~${p^3n = o(1)}$, ${\log(1/p) \cdot np^5  = o(1)}$ and~${p^{-1/2}/\log(1/p) \gg 1}$ imply that 
\begin{equation}
\min\biggcpar{\frac{(1-p^3)^{n/2}}{6C[3^2 \log(1/p)np]^{1/4} \cdot p}, \: \frac{p^{-1/2} (1-p^3)^{n/2}}{3^5 \log(1/p)}}
 \; > \; 1 ,
\end{equation}
which verifies assumptions~\eqref{eq:thm:lower:p:upper:0}--\eqref{eq:thm:lower:p:upper} of Theorem~\ref{thm:lower} for~$k=3$ and~$\tau=0.4$.  
Since~$p^{5/2}n/\log (1/p) = O(1)$, by invoking~\eqref{eq:thm:lower:chi} with~$k=3$ it follows that there is a constant~$c_2>0$ such that~whp
\begin{equation}\label{eq:lower:2}
\chi_c\bigpar{\Gnp} \: \ge \: c_2 \cdot \frac{p^{3/2}n}{\sqrt{\log n}} \qquad\qquad \text{if~$p \in [p^*_0,p_2]$.}
\end{equation}
Combining the maximum property~\eqref{eq:max1} with inequalities~\eqref{eq:lower:1} and~\eqref{eq:lower:2}, 
for~$p \in [p^*_0,p^*_1] \cup [p_1,p_2]$ it follows that the lower bound in~\eqref{eq:main1} holds without extra polylogarithmic factors. 
For the remaining range~$p \in [p^*_1,p_1]$, note that~$e^{np^2}p^{1/2}\log n=\Omega(1)$ implies that 
the right-hand side of~\eqref{eq:lower:2} is at least ${\Omega( (\log n)^{-1/2})}$ times the right-hand side of~\eqref{eq:lower:1}, 
so that inequality~\eqref{eq:lower:2} establishes the lower bound in~\eqref{eq:main1} with an extra factor of~${\Omega( (\log n)^{-1/2})}$.
\end{proof}

\subsection{Proof of Theorem~\ref{thm:main2}}\label{sec:proofmain2}
Gearing up towards the proof of Theorem~\ref{thm:main2}, it is routine to check that the maximum in~\eqref{eq:main2} satisfies
\begin{equation}\label{eq:max2}
\max\biggcpar{1, \; \frac{e^{-np^2}np}{\log n}, \; \min\biggcpar{\frac{p^{3/2}n}{\sqrt{\log n}}, \:  \frac{1}{p}}} 
= \begin{cases}
\Theta(1)& \text{if $p \in [0,p_0]$,} \\
\Theta\Bigpar{\frac{e^{-np^2}np}{\log n}} & \text{if $p \in [p_0,p_1]$,} \\
\Theta\Bigpar{\frac{p^{3/2}n}{\sqrt{\log n}}} & \text{if $p \in [p_1,p_2]$,} \\
\Theta\Bigpar{\frac{1}{p}} & \text{if $p \in [p_2,1]$.} 
\end{cases}
\end{equation}
In the below proof of Theorem~\ref{thm:main2} we shall also use the following convenient results of McDiarmid, Mitsche and Pra{\l}at, 
which follow from Theorem~1.3~(a) 
and Lemmas~3.2--3.3 in~\cite{MMP}. 
Here~$D(G)$ denotes the size of the smallest dominating set of the graph~$G$, 
i.e., the size of the smallest subset~$D \subseteq V(G)$ of the vertices 
such that every vertex~$v \in V(G)$ not in~$D$ is adjacent to at least one member of~$D$. 
\begin{lemma}\label{lem:previous}\cite{MMP} 
The following holds for any fixed~$\gamma>0$.  
\vspace{-0.25em}\begin{enumerate}[(i)]
\partopsep=0pt \topsep=0pt \parskip0pt \parsep0pt \itemsep0.25em
    \item\label{item:lem:chromatic} If~$n^{-1} \ll p \le n^{-1/2-\gamma}$, then whp~$3^{-1} np/\log(np) \le \chi_c(\Gnp) \le \chi(\Gnp) \le np/\log(np)$.  
    \item\label{item:lem:dominating} If~$n^{-1}(\log n) \ll p \le 1-\gamma$, then whp~$\chi_c(\Gnp) \le 1+D(\Gnp) \le 2 \log_{1/(1-p)}(n)$. 
		Furthermore, for all~${p=p(n) \in [0,1]}$ we always have $\chi_c(\Gnp) \le 1+D(\Gnp)$. 
\end{enumerate}
\end{lemma}
\begin{proof}[Proof of Theorem~\ref{thm:main2}] 
We first establish the lower and upper bounds in~\eqref{eq:main2} for~${p \in [0,p_2]}$. 
In view of the maximum property~\eqref{eq:max2}, 
note that the proof of Theorem~\ref{thm:main1} directly establishes the bounds in~\eqref{eq:main2} for~${p \in [p_0^*,p_2]}$. 
For~$p \in [p_0,p_0^*]$ this similarly follows from part~\eqref{item:lem:chromatic} of Lemma~\ref{lem:previous}. 
We now turn our attention to the range~${p \in [0,p_0]}$. 
In view of the maximum property~\eqref{eq:max2}, 
the trivial inequality~$\chi_c(\Gnp) \ge 1$ already establishes the lower bound in~\eqref{eq:main2}.
For the upper bound we shall exploit that the usual chromatic number is monotone increasing under the addition of edges.
Namely, by combining the trivial bound~${\chi_c(\Gnp) \le \chi(\Gnp)}$ with a natural coupling satisfying~${\Gnp \subseteq \Gnpp{p_0}}$, 
using monotonicity of~$\chi(\Gnp)$ and part~\eqref{item:lem:chromatic} of Lemma~\ref{lem:previous} 
we infer that whp~$\chi_c(\Gnp) \le \chi(\Gnp) \le \chi(\Gnpp{p_0}) = o(\log n)$, 
establishing the upper bound in~\eqref{eq:main2} for~$p \in [0,p_0]$. 

We next establish the upper bound in~\eqref{eq:main2} for~$p \in [p_2,1]$.
Note that ${\log_{1/(1-p)}(n)}={(\log n)/[p(1+O(p))]}$ for~$p \le 1/2$, say. 
In view of the maximum property~\eqref{eq:max2}, 
it follows that part~\eqref{item:lem:dominating} of Lemma~\ref{lem:previous} 
establishes the upper bound in~\eqref{eq:main2} for~$p \in [p_2,1/2]$, say. 
For the remaining range~$p \in [1/2,1]$ we shall exploit that the size of the smallest dominating set is monotone decreasing under the addition of edges.
Namely, by combining a natural coupling satisfying~${\Gnp \supseteq \Gnpp{1/2}}$ with monotonicity of~$D(\Gnp)$, 
using part~\eqref{item:lem:dominating} of Lemma~\ref{lem:previous} we infer that whp~$\chi_c(\Gnp) \le {1+D(\Gnp)} \le {1+D(\Gnpp{1/2})} \le {2\log_2(n)}$, 
establishing the upper bound in~\eqref{eq:main2} for~$p \in [1/2,1]$. 

Finally, we turn to the more involved lower bound in~\eqref{eq:main2} for~$p \in [p_2,1]$. 
In view of the maximum property~\eqref{eq:max2} and the trivial inequality~$\chi_c(\Gnp) \ge 1$, 
it remains to consider the range~$p_2 \le p =o(1)$, say. We~define 
\begin{equation}\label{eq:defk}
k:=\bigceil{\log_{1/p}(n)} ,
\end{equation}
for which is is easy to check that~$3 \le k =o(\log n)$. 
We now verify assumptions~\eqref{eq:thm:lower:p:upper:0}--\eqref{eq:thm:lower:p:upper} of Theorem~\ref{thm:lower} with~$\tau=0.4$, say.
Using~$n^{-2/5} \le p_2 \le p=o(1)$ together with $\log(1/p) \le \log n$ and~$np^k \le 1$, it follows that 
\begin{equation}
\frac{np^2(1-p^k)^{n(k-2)/(k-1)}}{k^5 \log(1/p)} 
\: \ge \:  
\frac{n^{1/5} \cdot e^{-2np^k}}{(\log n)^6} \ge n^{1/5-o(1)} > 1 ,
\end{equation}
which verifies assumption~\eqref{eq:thm:lower:p:upper}. 
Using~$k^2\log(1/p)/(np) \le (\log n)^3/n^{3/5} < n^{-2/5}$ and~$np^k \le 1$ together with~$2 \le k-1 =o(\log n)$,
it similarly follows that 
\begin{equation}
\frac{(1-p^k)^{n/(k-1)}}{6C\bigsqpar{k^2 \log(1/p) np}^{1/2(k-1)} \cdot p}
\: \ge \:  
\frac{e^{-1}}{6C\bigsqpar{n^{-2/5} (np^{k})^2}^{1/2(k-1)}} 
\ge 
\frac{n^{1/5(k-1)}}{6Ce}
> 1 ,
\end{equation}
which verifies assumption~\eqref{eq:thm:lower:p:upper:0}. 
Using~$p=o(1)$ and~$np^k \le 1$, by invoking~\eqref{eq:thm:lower:chi} 
there is a constant~${c_3>0}$ such that~whp 
\begin{equation}\label{eq:lower:3}
\chi_c\bigpar{\Gnp} \; \ge \; c_3 \cdot \min\Biggcpar{\frac{1}{p}, \: \frac{p^{k/2}n}{\bigsqpar{k!k\log(1/p)}^{1/(k-1)} }}
\qquad\qquad \text{if~$p_2 \le p = o(1)$.}
\end{equation}
When~$k \ge 5$ we have $p > n^{-1/4}$, in which case it is easy to see 
that~$np^{k} \ge p$ and~$k!k \le k^{k+1}$ as well as~$\max\{k,\log(1/p)\} \le \log n$ and~$p^3 n \ge n^{1/4}$ 
together yield that 
\begin{equation}
\frac{p^{k/2+1}n}{\bigsqpar{k!k\log(1/p)}^{1/(k-1)}} 
\: \ge \:
\frac{\sqrt{p^{3}n}}{k^{O(1)} (\log n)^{O(1)}} 
\ge n^{1/8-o(1)}
> 1 ,
\end{equation}
which implies that the minimum in~\eqref{eq:lower:3} is of order~$1/p$. 
This order implication also holds when~${k=3}$, since then ${p^{k/2+1}n/[\log(1/p)]^{1/(k-1)}} \ge {(p_2)^{5/2}n/\sqrt{\log n}} = 1$.
In the remaining case~${k=4}$ this implication about the order of the minimum in~\eqref{eq:lower:3} also holds when~${p \ge (\log n)^{1/9}n^{-1/3}}$, 
since then ${p^{k/2+1}n/[\log(1/p)]^{1/(k-1)}} \ge {p^3n/(\log n)^{1/3}} \ge 1$. 
Putting things together, inequality~\eqref{eq:lower:3} implies that whp~${\chi_c(\Gnp) =\Omega(1/p)}$ 
when~${k \neq 4}$ or~${p \ge (\log n)^{1/9}n^{-1/3}}$.
To complete the proof of the lower bound in~\eqref{eq:main2}, 
in view of the maximum property~\eqref{eq:max2} it thus remains to consider the case~$k=4$ when~${n^{-1/3} < p < (\log n)^{1/9}n^{-1/3}}$.
In this case the weaker bound~$p^{k/2+1}n/[\log(1/p)]^{1/(k-1)} \ge p^3n/(\log n)^{1/3} \ge (\log n)^{-1/3}$ holds, 
which by inequality~\eqref{eq:lower:3} implies that whp~${\chi_c(\Gnp) =\Omega\bigpar{(\log n)^{-1/3}/p}}$, 
completing the proof of the lower bound in~\eqref{eq:main2}.
\end{proof}

The above proof shows that there are constants~$c_0,C_0>0$ such that whp~$c_0/p \le \chi_c(\Gnp) \le C_0 (\log n)/p$ 
when~${(\log n)^{1/5}n^{-2/5} \le p \le n^{-1/3}}$ or~${(\log n)^{1/9}n^{-1/3} \le p = o(1)}$, 
which in contrast to~\cite[Theorem~1.3~(h)]{MMP} also gives a non-trivial lower bound on~$\chi_c(\Gnp)$ for~${p=n^{-o(1)}}$. 
We believe that~$\chi_c(\Gnp) = \Theta((\log n)/p)$ should be the correct order of magnitude, 
see~\eqref{eq:conj1} and Conjecture~\ref{conj1} in the next section.

\section{Concluding remarks}\label{sec:conclusion}
The main remaining open problem is to determine the typical order of magnitude 
of the clique chromatic number~$\chi_c(\Gnp)$ 
for all edge-probabilities~$p=p(n)$ of interest, 
i.e., to close the gaps in Theorem~\ref{thm:main1} and~\ref{thm:main2}.
\begin{problem}\label{prm:open}
Determine the whp order of magnitude of~$\chi_c(\Gnp)$ 
when~$p \gg n^{-1}$ and $1-p = \Omega(1)$. 
\end{problem}
In view of previous work~\cite{MMP,AK17,DK}, Lemma~\ref{lem:previous} and Theorem~\ref{thm:main1}, 
it is tempting to conjecture that the bounds of Theorem~\ref{thm:main2} can be further refined to the with high probability bound
\begin{equation}\label{eq:conj1}
\displaystyle\chi_c(\Gnp)
= \Theta\biggpar{\max\biggcpar{\displaystyle 1, \; \displaystyle\frac{e^{-np^2}np}{\log(np)}, \; \min\biggcpar{\displaystyle \frac{p^{3/2}n}{\sqrt{\log n}}, \: \displaystyle \log_{1/(1-p)}(n)}}} ,
\end{equation}
though there might be extra complications near the transition probabilities~${p = n^{-x+o(1)}}$ with~${x \in \{1/2,2/5,1/3\}}$, 
around which the optimal lower bound strategies appear to change.
In terms of supporting evidence, \eqref{eq:conj1} holds for constant~${p \in (0,1)}$ due to work of Alon and Krivelevich~\cite{AK17} and Demidovich and Zhukovskii~\cite{DK}.
Furthermore, \eqref{eq:conj1} holds for most ranges of~$p=p(n)$ with~${n^{-1} \ll p \le n^{-2/5-o(1)}}$ due to Theorem~\ref{thm:main1}, Lemma~\ref{lem:previous} and~\mbox{\cite[Theorem~1.3]{MMP}}, 
noting that~${\log_{1/(1-p)}(n) \sim (\log n)/p}$ when~${p=o(1)}$.  
In view of~\eqref{eq:conj1}, the next step towards~Problem~\ref{prm:open} is thus the following conjecture, 
whose upper bound holds by~\eqref{item:lem:dominating} of Lemma~\ref{lem:previous}.
\begin{conjecture}\label{conj1}
If~$p=p(n)$ satisfies~${(\log n)^{3/5}n^{-2/5} \ll p = o(1)}$, then whp~${\chi_c(\Gnp) =  \Theta(\log n/p)}$. 
\end{conjecture}

\small
\bibliographystyle{plain}

\normalsize

\end{document}